\documentclass{article}

\usepackage{graphicx,amsmath,latexsym,amssymb,amsthm}

\newcounter{Theorem}

\theoremstyle{plain}
\newtheorem{thm}[Theorem]{\bf Theorem}
\newtheorem{prop}[Theorem]{\bf Proposition}
\newtheorem{cor}[Theorem]{\bf Corollary}
\newtheorem{lemma}[Theorem]{\bf Lemma}

\theoremstyle{remark}
\newtheorem{remark}[Theorem]{\bf Remark}

\theoremstyle{definition}

\numberwithin{Theorem}{section}

\newcommand{\HH}{\mathcal{H}}

\newcommand{\R}{\mathbb{R}}

\newcommand{\N}{\mathbb{N}}

\newcommand{\leqs}{\leqslant}
\newcommand{\geqs}{\geqslant}

\newcommand{\ep}{\epsilon}
\newcommand{\ld }{\lambda}

\newcommand{\sm}{\,\sigma\,}

\newcommand{\norm}[1]{\lVert #1 \rVert}

\begin{document}

\thispagestyle{empty}

\setcounter{page}{1}

\vspace*{-1.5cm}


%

\label{1st-page}

\renewcommand{\thefootnote}{\fnsymbol{footnote}}

\begin{center}
{\Large{\textbf{The Normed Ordered Cone of Operator}}} \\[0.2cm]
{\Large{\textbf{Connections}}} \\[0.4cm]
{Pattrawut Chansangiam{\footnote{Corresponding author}},
        Wicharn Lewkeeratiyutkul{}} \\[0.2cm] 
\end{center}

%

\renewcommand{\thefootnote}{\arabic{footnote}}

\bigskip

\noindent
\textbf{Abstract:\ } 
A connection in Kubo-Ando sense is a binary operation for positive operators on a Hilbert space satisfying
the monotonicity, the transformer inequality and the continuity from above.
A mean is a connection $\sigma$ such that $A \sigma A =A$ for all positive operators $A$.
In this paper, we consider the interplay between the cone of connections,
the cone of operator monotone functions on $\R^+$ and the cone of finite Borel measures on $[0,\infty]$.
The set of operator connections is shown to be isometrically order-isomorphic, as normed ordered cones,
to the set of operator monotone functions on $\R^+$.
This set is isometrically isomorphic, as normed cones,
to the set of finite Borel measures on $[0,\infty]$.
It follows that the convergences of the sequence of connections, the sequence of their representing functions
and the sequence of their representing measures are equivalent.
In addition, we obtain characterizations for a connection to be a mean.
In fact, a connection is a mean if and only if it has norm $1$.
    
\vspace{0.6cm}

\noindent
\textbf{Keywords:\ } operator connection, operator mean, operator monotone function  

\medskip

\noindent
\textbf{2010 Mathematics Subject Classification:\ } 47A63, 47A64 

\bigskip

\section{Introduction}

A general theory of operator means 
was given by Kubo and Ando \cite{Kubo-Ando}. 
Denote by $B(\HH)$ the von Neumann algebra of bounded linear operators on a complex Hilbert space $\HH$
and $B(\HH)^+$ its positive cone.
A \emph{connection} is a binary operation $\sm$ assigned to each pair of positive operators
such that for all $A,B,C,D \geqs 0$:
\begin{enumerate}
	\item[(M1)] \emph{monotonicity}: $A \leqs C, B \leqs D \implies A \sm B \leqs C \sm D$
	\item[(M2)] \emph{transformer inequality}: $C(A \sm B)C \leqs (CAC) \sm (CBC)$
	\item[(M3)] \emph{continuity from above}:  for $A_n,B_n \in B(\HH)^+$,
                if $A_n \downarrow A$ and $B_n \downarrow B$,
                 then $A_n \sm B_n \downarrow A \sm B$. 
                 Here, $A_n \downarrow A$ indicates that $A_n$ is a decreasing sequence
                 converging strongly to $A$.
\end{enumerate}
This definition is modeled from the notion of the parallel sum introduced in \cite{Anderson-Duffin}
for analzing multiport electrical networks.
A \emph{mean} is a connection $\sigma$ such that $A \sm A =A$ for all $A \geqs 0$ or, equivalently, $I \sm I = I$.
Here are examples of means in practical usage:  
	\begin{itemize}
		\item	arithmetic mean: $A \triangledown B = (A + B)/2$
		\item	geometric mean \cite{Ando78, Ando}: 
    			$A \# B =  A^{1/2} 
    			({A}^{-1/2} B  {A}^{-1/2})^{1/2} {A}^{1/2}$
		\item	harmonic mean: $A \,!\, B = 2(A^{-1} + B^{-1})^{-1}$ 
		\item	logarithmic mean: $(A,B) \mapsto A^{1/2} f (A^{-1/2} B A^{-1/2}) A^{1/2}$ where $f(x) = (x-1)/\log x$.
	\end{itemize}

For connections $\sigma$ and $\eta$ on $B(\HH)^+$, we define
\begin{align*}
    \sigma + \eta &: B(\HH)^+ \times B(\HH)^+ \to B(\HH)^+ : (A,B) \mapsto (A \,\sigma\, B) + (A \,\eta\, B), \\
    k\sigma &: B(\HH)^+ \times B(\HH)^+ \to B(\HH)^+ : (A,B) \mapsto k(A \,\sigma\, B), \quad k \in \R^+.
\end{align*}
Denote by $C(B(\HH)^+)$ the set of connections on $B(\HH)^+$.
Define a partial order $\leqs$ for connections on $B(\HH)^+$ by $\sigma_1 \leqs \sigma_2$ 
if $A \,\sigma_1\, B \leqs A \,\sigma_2\, B$ for all $A,B \in B(\HH)^+$. 
It is straightforward to show that the set $C(B(\HH)^+)$ is an ordered cone in which the neutral element is 
the zero connection $0: (A,B) \mapsto 0$. 
This cone is pointed (i.e. $\sigma \geqs 0$ for all $\sigma$) and order cancellative.

A major tool in Kubo-Ando theory of connections and means is the class of operator monotone functions. This concept was introduced in \cite{Lowner}; see more information in \cite{Bhatia,Donoghue,Hiai-Yanagi}. 
Recall that continuous real-valued function $f$ on an interval $I$ is called an \emph{operator monotone function} if
for all Hilbert spaces $\HH$ and for all Hermitian operators $A,B \in B(\HH)$ whose spectra are contained in $I$, 
we have
\begin{align*}
	A \leqs B \implies f(A) \leqs f(B).
\end{align*} 
Denote by $OM(\R^+)$ the set of operator monotone functions from $\R^+=[0,\infty)$ to itself.
This set is a cone under usual addition and scalar multiplication 
in which the zero function $0: x \mapsto 0$ is the neutral element.
The partial order on $OM(\R^+)$ is defined pointwise.
This cone becomes an ordered cone which is pointed and order cancellative.

A major result in Kubo-Ando theory is a one-to-one correspondence between connections on $B(\HH)^+$ and 
operator monotone functions on $\R^+$ as follows:

\begin{thm}[\cite{Kubo-Ando}] \label{thm: connection and operat mon func}
	Given a connection $\sigma$, there is a unique operator monotone function $f: \R^+ \to \R^+$ satisfying
				\begin{align*}
    			f(x)I = I \sm (xI), \quad x \in \R^+.
				\end{align*}
	Moreover, the map $\sigma \mapsto f$ is an affine order isomorphism.
	In addition, $\sigma$ is a mean if and only if $f(1)=1$.
\end{thm}
We call $f$ the \emph{representing function} of $\sigma$.
There is also a one-to-one correspondence between connections and finite Borel measures on $[0,\infty]$
given by the following integral representation:

\begin{thm}[\cite{Kubo-Ando}] \label{thm: connection and Borel measure on [0,infty]}
	Given a connection $\sigma$, there is a unique finite Borel measure $\mu$ on $[0,\infty]$ such that
				\begin{align}
    			A \sm B = \int_{[0,\infty]} \frac{\ld+1}{2 \ld} (\ld A \,!\,B)\, d\mu(\ld), \quad A,B \geqs 0.
    			\label{eq: formula of connection}
				\end{align}
			Moreover, the map $\sigma \mapsto \mu$ is an affine isomorphism.
			In addition, $\sigma$ is a mean if and only if $\mu([0,\infty]) =1$.
\end{thm}
The measure $\mu$ in this theorem is called the \emph{representing measure} of $\sigma$.
Here, the cone of finite Borel measures on $[0,\infty]$, denoted by $BM([0,\infty])$,
is equipped with the usual algebraic operations and pointwise order.
The cone $BM([0,\infty])$ is an ordered cone in which the zero measure 
is the neutral element.

In this paper, we investigate structures of the cone of connections in relation with the cone 
of operator monotone functions on $\R^+$ and the cone of finite Borel measures on $[0,\infty]$.
We define a norm for a connection in such a way that the set of operator connections becomes
a normed ordered cone.
On the other hand, the cone of operator monotone functions on $\R^+$
and the cone of finite Borel measures on $[0,\infty]$ are equipped with suitable norms.
The set of operator connections is shown to be isometrically order-isomorphic, as normed ordered cones,
to the set of operator monotone functions on $\R^+$ via the map sending a connection to its representing function.
This set is isometrically isomorphic, as normed cones,
to the set of finite Borel measures on $[0,\infty]$ via the map sending a connection to its representing measure.
It follows that the convergences of the sequence of connections, the sequence of their representing functions
and the sequence of their representing measures are equivalent.
In addition, we obtain characterizations for a connection to be a mean.
In fact, a connection is a mean if and only if it has norm $1$.

\section{Norms for connections, operator monotone functions and Borel measures}

In this section, we consider topological structures of the cone of connections, 
the cone of operator monotone functions on $\R^+$
and the cone of finite Borel measures on $[0,\infty]$.
In fact, there are norms equipped naturally on these cones.

    Recall that a \emph{normed cone} is a cone $C$ equipped with a function $\norm{\cdot} : C \to \R^+$ such that
    for each $x,y \in C$ and $k \in \R^+$,
    \begin{enumerate}
        \item[(i)]   $\norm{x}=0 \implies x=0$,
        \item[(ii)]   $\norm{kx} = k \norm{x}$,
        \item[(iii)]   $\norm{x+y} \leqs \norm{x}+ \norm{y}$.
    \end{enumerate}
    A \emph{normed ordered cone} is an ordered cone $(C, \leqs)$ which is also a normed cone such that
    for each $x,y \in C$, $x \leqs y \implies \norm{x} \leqs \norm{y}$.

We define a function $\norm{\cdot}: C(B(\HH)^+) \to \R^+$ by 
\begin{align*}
	\norm{\sigma} = \sup\left\{ \norm{A \sm B}: A,B \geqs 0, \norm{A}=\norm{B}=1 \right\}
\end{align*}
for each connection $\sigma$. 

Recall that each connection $\sigma$ on $B(\HH)^+$ induces a unique connection $\tilde{\sigma}$
on $\R^+$ such that
\begin{align*}
	(xI) \,\sigma\, (yI) = (x \,\tilde{\sigma}\, y)I, \quad x,y \in \R^+.
\end{align*} 
A connection and its induced connection may be written by the same notation.

\begin{lemma}(\cite{Arlinskii}) \label{lem: norm of A sm B}
	For each connection $\sigma$, we have
	$
		\norm{A \sm B} \leqs \norm{A} \sm\norm{B}
	$
	for all $A,B \geqs 0$.
\end{lemma}

\begin{prop} \label{prop: def of norm of sm}
	For each connection $\sigma$, we have 
	\begin{align*}
		\norm{\sigma}
			\:&=\: \sup \left\{ \frac{\norm{A \sm A}}{\norm{A}}: A >0 \right\} \\
			\:&=\: \sup \left\{ \norm{A \sm A}: A\geqs 0, \norm{A} = 1 \right \} \\	
			\:&=\: \frac{\norm{ A \sm A}}{ \norm{A} } \quad \text{ for any } A>0 \\
			\:&=\: \norm{I \sigma I}.  
	\end{align*}
\end{prop}
\begin{proof}
	Clearly, $\norm{\sigma} \geqs \norm{I \sigma I}$.
	For each $A,B \geqs 0$ with $\norm{A} = \norm{B} = 1$, it follows from 
	Lemma \ref{lem: norm of A sm B} that
	\begin{align*}
		\norm{A \sm B} \leqs \norm{A} \sm \norm{B} 
		= 1 \,\sigma\, 1 = \norm{(1 \,\sigma\, 1) I} = \norm{I \sm I}.
	\end{align*}
	Hence, $\norm{\sigma} \leqs \norm{I \sm I}$.
	For each $A>0$, we have $\displaystyle \norm{\frac{1}{\norm{A}} A} =1 $ and hence
	\begin{align*}
		\norm{\sigma}
		\geqs \norm{ \frac{1}{\norm{A}}A \sm \frac{1}{\norm{A}}A } 
		= \norm{\frac{1}{\norm{A}} (A \sm A)} 
		= \frac{\norm{A \sm A}}{\norm{A}}  
	\end{align*}
	On the other hand, for $A>0$ we have
	\begin{align*}
		\norm{I \sm I} 
		&= \norm{ A^{-1/2} (A \sm A) A^{-1/2} }  \\
		&\leqs \norm{A^{-1/2}} \norm{A \sm A} \norm{A^{-1/2}} \\
		&= \norm{A}^{-1/2} \norm{A \sm A} \norm{A}^{-1/2}.
	\end{align*}
	Thus, $\norm{\sigma} = \norm{A \sm A}/\norm{A}$ for any $A>0$.
\end{proof}

\begin{prop}
    The pair $\left(C(B(\HH)^+), \norm{\cdot}\right)$ is a normed ordered cone.
\end{prop}
\begin{proof}
	For each $\sigma, \eta \in C(B(\HH)^+)$ and $k \in \R^+$, by Proposition \ref{prop: def of norm of sm} we have 
	\begin{align*}
		\norm{k \sigma} &= \norm{I \,(k\sigma)\, I} = \norm{k(I \sigma I)} = k \norm{I \sigma I} = k \norm{\sigma}, \\
		\norm{\sigma + \eta} &= \norm{I \,(\sigma+\eta)\, I} = \norm{(I \sigma I) + (I \eta I)} 
			\leqs \norm{I \sigma I} + \norm{I \eta I} = \norm{\sigma} + \norm{\eta}.
	\end{align*}
	Suppose now that $\norm{\sigma} = 0$, i.e. $I \sigma I = 0$. 
	For each projection $P$, we have $P \leqs I$ and hence $I \sigma P \leqs I \sigma I =0$, i.e. $I \sigma P = 0$.
	Similarly, $(xI) \sigma I = 0$ for each $x \in [0,1]$.
	Then for each $x>1$, 
	\begin{align*}
		I \sigma (xI) = x\left( \frac{1}{x}I \,\sigma\, I \right) = 0.
	\end{align*}
	Consider $A \in B(\HH)^+$ in the form $A = \sum_{i=1}^m \ld_i P_i$ where $\ld_i>0$ and $P_i$'s are projections such that
	$P_i P_j = 0$ for $i \neq j$ and $\sum_{i=1}^m P_i = I$.
	We have
	\begin{align*}
		I \sm A = \sum_{i=1}^m (I \sm A)P_i = \sum P_i \sm AP_i 
			= \sum P_i \sm \ld_i P_i = \sum P_i (I \sm \ld_i I) = 0.
	\end{align*}
	For general $A \in B(\HH)^+$, let $\{A_n\}$ be a sequence of invertible positive operators
	such that $A_n \downarrow A$.
	Then $I \sm A = \lim_{n \to \infty} I \sm A_n = 0$ for all $A \geqs 0$.
	Hence, for $A,B \in B(\HH)^+$,
	\begin{align*}
		A \sm B = \lim_{ \ep \downarrow 0} A_{\ep} \sm B 
			= \lim_{\ep \downarrow 0} A_{\ep}^{1/2} (I \sm A_{\ep}^{-1/2} B A_{\ep}^{-1/2}) A_{\ep}^{1/2} = 0,
	\end{align*}
	here $A_{\ep} \equiv A+\ep I$. Thus $\sigma = 0$.
	
	If $\sigma \leqs \eta$, then $\norm{\sigma} = \norm{I \sm I} \leqs \norm{I \,\eta\, I} = \norm{\eta}$ since
	$I \sm I \leqs I \,\eta\, I$.
\end{proof}

Recall that a function $f$ from a cone $C$ into a cone $D$ is called \emph{linear} or \emph{affine} if 
$f(rx + sy) = r f(x) + s f(y)$
for each $x,y \in C$ and $r,s \in \R^+$.

Define a function $\norm{\cdot}: OM(\R^+) \to \R^+$ by 
$
	\norm{f} 
		= f(1)
$
for each $f \in OM(\R^+)$. 
\begin{prop} \label{prop: OM(R+) is ordered normed cone and norm is linear}
	The pair $(OM(\R^+), \norm{\cdot})$ is a normed ordered cone. Moreover, the function $\norm{\cdot}$ is linear.
\end{prop}
\begin{proof}
	The only non-trivial part is to show that $\norm{f}=0$ implies $f=0$.
	Consider $f \in OM(\R^+)$ such that $f(1)=0$. Suppose that there is an $a>0$ such that $f(a)=0$.
	Then $f(x) = 0$ for $0 \leqs x \leqs a$. 
	Since $f \in OM(\R^+)$, $f$ is a concave function by \cite{Hansen-Pedersen}. 
	The concavity of $f$ implies that $f=0$.
\end{proof}

Assign to each measure $\mu \in BM([0,\infty])$ its total variation: 
\begin{align*}
	\norm{\mu} = \mu([0,\infty]) < \infty.
\end{align*}

\begin{prop}
	The pair $(BM([0,\infty]), \norm{\cdot})$ is a normed ordered cone. Moreover, the function $\norm{\cdot}$ is linear.
\end{prop}

Given any normed cone, we can equip it with a topology as follows.

\begin{prop} \label{prop: topology on normed cone}
	Let $(C, \norm{\cdot})$ be a normed cone. Then
	\begin{enumerate}
		\item[(1)]	the function $d: C \times C \to \R^+$, $d(x,y) = \big|\norm{x}-\norm{y}\big|$ is 
					a pseudo metric; in particular, $C$ is a $1$st-countable topological space with respect to the topology 
					induced by $d$.
		\item[(2)] the functions $\norm{\cdot}$ and $d$ are continuous,
					where the topology on $C \times C$ is given by the product topology.
		\item[(3)]	$C$ becomes a topological cone in the sense that the addition and
					the scalar multiplication are continuous.
	\end{enumerate}
\end{prop}
\begin{proof}
	The proof is similar to the case of normed linear spaces. Note that the topology induced by a pseudo metric
	satisfies the $1$st-countability axiom. In this topology, a function is continuous if and only if 
	it is sequentially continuous. 
\end{proof}

Hence the cones $C(B(\HH)^+)$, $OM(\R^+)$ and $BM([0,\infty])$ are toplological cones.

\section{The isomorphism theorem}

In this section, we establish isomorphisms between the cone of connections, 
the cones of operator monotone functions on $\R^+$
and the cone of finite Borel measures on $[0,\infty]$.  

Recall the following terminology.
	A function $\varphi: C \to D$ between normed cones is called an \emph{isomorphism} if
	it is a continuous linear bijection whose inverse is continuous. 
	By an \emph{isometry}, we mean a linear function $\phi: C \to D$ such that 
	$\norm{\phi(c)} = \norm{c}$ for all $c \in C$.
	Note that every isometry between normed cones is continuous and injective.
The inverse of an isometry is an isometry.
	If $\varphi : C \to D$ is an isomorphism which is also an isometry, we say that $\varphi$ is an 
	\emph{isometric isomorphism} and $C$ is said to be \emph{isometrically isomorphic} to $D$.

	Let $C$ and $D$ be normed ordered cones. A function $\varphi: C \to D$ is called an \emph{order isomorphism} if
	it is an isomorphism (between normed cones) such that $\varphi$ and $\varphi^{-1}$ are order-preserving. 
	If, in addition, $\varphi$ is an isometry, we say that $\varphi$ is an 
	\emph{isometric order-isomorphism} and $C$ is said to be \emph{isometrically order-isomorphic} to $D$. 

\begin{thm} \label{thm: isometric order isomorphism between connections and OM}
\begin{enumerate}
		\item[(1)]	The normed ordered cones $C(B(\HH)^+)$ and $OM(\R^+)$ are isometrically order-isomorphic via the isometric
					order-isomorphism $\sigma \mapsto f_{\sigma}$, where $f_{\sigma}$ is the representing function
					of $\sigma$. 
		\item[(2)]	The normed cones $C(B(\HH)^+)$ and $BM([0,\infty])$ are isometrically isomorphic via the isometric
					isomorphism $\sigma \mapsto \mu_{\sigma}$, where $\mu_{\sigma}$ is the representing measure
					of $\sigma$.
\end{enumerate}
\end{thm}
\begin{proof}
	 The function $\Phi: \sigma \mapsto f_{\sigma}$ is an order isomorphism by 
	 Theorem \ref{thm: connection and operat mon func}. 
	 For each connection $\sigma$, since $f_{\sigma}(1)I = I \sm I$, we have
	 \begin{align*}
	 		\norm{\Phi(\sigma)} = \norm{f_{\sigma}} = f_{\sigma} (1) = \norm{I \sm I} = \norm{\sigma}.
	 \end{align*}
	 The function $\Psi: \sigma \mapsto \mu_{\sigma}$ is an isomorphism by Theorem 
	 \ref{thm: connection and Borel measure on [0,infty]}. 
	 For each connection $\sigma$, we have
	 \begin{align*}
	 		\norm{\Psi(\sigma)} = \norm{\mu_{\sigma}} = \mu([0,\infty]) = \norm{I \sm I} = \norm{\sigma}
	 \end{align*}
	 since $I \sm I = \int_{[0,\infty]} \frac{\ld+1}{2\ld} (\ld I \,!\,I) \,d\mu(\ld) = \mu([0,\infty])I$.
\end{proof}

\begin{remark} 
	Even though the map $\mu \mapsto \sigma$, sending finite Borel measures to their associated connections, is order-preserving,
	the inverse map $\sigma \mapsto \mu$ is not order-preserving in general.
	For example, the representing measures of the harmonic mean $!$
	and the arithmetic mean $\triangledown$
	are given by $\delta_1$ and $(\delta_0 + \delta_{\infty})/2$, respectively.
	Here, $\delta_x$ is the Dirac measure at $x$.
	We have $! \leqs \triangledown$ but $\delta_1 \not\leqslant (\delta_0 + \delta_{\infty})/2$.
\end{remark}


\begin{cor}
	The function $\norm{\cdot}$ on $C(B(\HH)^+)$ is linear.
\end{cor}
\begin{proof}
	It follows from the fact that the map $\sigma \mapsto f_{\sigma}$ is an isometric isomorphism  
	and the norm on $OM(\R^+)$ is linear.
\end{proof}

We obtain the following characterizations of a mean as follows.

\begin{cor} \label{cor: TFAE for a mean}
	The followings are equivalent for a connection $\sigma$:
	\begin{enumerate}
		\item[(i)]	$\sigma$ is a mean;
		\item[(ii)] $\norm{\sigma} = 1$;
		\item[(iii)] $\norm{A \sm A} = \norm{A}$ for all $A \geqs 0$;
		\item[(iv)] $\norm{A \sm A} = \norm{A}$ for some $A >0$.
	\end{enumerate}
\end{cor}
\begin{proof}
	It follows from Proposition \ref{prop: def of norm of sm}, 
	Theorem \ref{thm: isometric order isomorphism between connections and OM} 
	and the fact that a connection is a mean if and only if its representing function (measure) is normalized.
\end{proof}

From the equivalence (i)-(ii) in this corollary, a mean is a normalized connection. 
Every mean arises as a normalization of a nonzero connection.
The convex set of means is the unit sphere in the cone of connections.

\begin{cor}
	The limit of a sequence of means is a mean.
\end{cor}
\begin{proof}
	Use the fact that the norm for connections is continuous by Proposition \ref{prop: topology on normed cone} and
	the norm of a mean is $1$ by Corollary \ref{cor: TFAE for a mean}.
\end{proof}

The topologies of the cones $C(B(\HH)^+)$, $OM(\R^+)$ and $BM([0,\infty])$ are compatible with
the isometric isomorphisms $\sigma \mapsto f_{\sigma}$ and $\sigma \mapsto \mu_{\sigma}$ in 
Theorem \ref{thm: isometric order isomorphism between connections and OM} as follows.

\begin{cor}
	For each $n \in \N$, let $\sigma_n$ be a connection with representing function $f_n$ and representing measure $\mu_n$.
	Then the followings are equivalent for a connection $\sigma$ with representing
	function $f$ and representing measure $\mu$:
	\begin{enumerate}
		\item[(i)]	$\sigma_n \to \sigma$;
		\item[(ii)]	$f_n \to f$;
		\item[(iii)]	$\mu_n \to \mu$.
	\end{enumerate}
\end{cor}

%


\vspace{0.6cm}

\begin{small}
\noindent
Pattrawut Chansangiam \\
Department of Mathematics, Faculty of Science\\
King Mongkut's Institute of Technology Ladkrabang \\
Bangkok 10520, Thailand \\
Email: \texttt{kcpattra@kmitl.ac.th}
\end{small}

\bigskip

\begin{small}
\noindent
Wicharn Lewkeeratiyutkul \\
Department of Mathematics and Computer Science, \\
Faculty of Science, Chulalongkorn University, \\
Bangkok 10330, Thailand \\
Email: \texttt{Wicharn.L@chula.ac.th}
\end{small}

\label{last-page}

\end{document}